\documentclass[11pt]{article}
\usepackage{enumerate}
\usepackage{amssymb,a4wide,latexsym,makeidx,epsfig,fleqn}
\usepackage{float}
\usepackage{amsthm}
\usepackage{amsmath}
\usepackage{enumerate}
\usepackage{color}
\usepackage{graphicx}
\usepackage{multirow}
\usepackage{makecell}
\newtheorem{theorem}{Theorem}[section]
\newtheorem{lemma}[theorem]{Lemma}
\newtheorem{corollary}[theorem]{Corollary}

\newtheorem{problem}[theorem]{Problem}
\newtheorem{claim}{Claim}

\newcommand{\tabincell}[2]{\begin{tabular}{@{}#1@{}}#2\end{tabular}}
\allowdisplaybreaks

\begin{document}
\textwidth 150mm \textheight 225mm

\title{Sharp bounds for Laplacian spectral moments \\ of digraphs with a fixed dichromatic number\thanks{\small Supported by the National Natural Science Foundation of China (Nos. 11871398, 12271439) and the China Scholarship Council (No. 202106290009).}}

\author{{Xiuwen Yang$^{a,b}$, Hajo Broersma$^{b,}$\footnote{Corresponding author.}, Ligong Wang$^{a}$}\\
{\small $^{a}$ School of Mathematics and Statistics,}\\{\small Northwestern Polytechnical University, Xi'an, Shaanxi 710129, P.R. China.}\\ {\small $^{b}$ Faculty of Electrical Engineering, Mathematics and Computer Science,}\\{\small University of Twente, P.O. Box 217, 7500 AE Enschede, The Netherlands.}\\{\small E-mail: yangxiuwen1995@163.com, h.j.broersma@utwente.nl, lgwangmath@163.com}}
\date{}
\maketitle
\begin{center}
\begin{minipage}{135mm}
\vskip 0.3cm
\begin{center}
{\small {\bf Abstract}}
\end{center}
{\small
The $k$-th Laplacian spectral moment of a digraph $G$ is defined as $\sum_{i=1}^n \lambda_i^k$, where $\lambda_i$ are
the eigenvalues of the Laplacian matrix of $G$ and $k$ is a nonnegative integer. For $k=2$, this invariant is better known as the Laplacian energy of $G$. We extend recently published results by characterizing the digraphs which attain the minimal and maximal Laplacian energy within classes of digraphs with a fixed dichromatic number. We also determine sharp bounds for the third Laplacian spectral moment within the special subclass which we define as join digraphs. We leave the full characterization of the extremal digraphs for $k\ge 3$ as an open problem.

}
\end{minipage}
\end{center}

\section{Introduction}

Before we give some essential terminology and notation, we start with a brief introduction, including some background and motivation for the presented results.

The domain of our results is algebraic graph theory, in particular the popular subarea of chemical graph theory which is based on concepts like the eigenvalues of matrices associated with graphs, and graph invariants involving these eigenvalues. One of the oldest concepts in this area is the notion of graph energy, which goes back to the late 1970s, when it was introduced by Gutman~\cite{Gut1} as the sum of the absolute values of the eigenvalues of the adjacency matrix of a graph. More recently, in 2006,  Gutman and Zhou~\cite{GuZh} and Lazi\'c~\cite{La} independently defined different versions of the Laplacian energy of a graph, where the version of Lazi\'c is defined as the sum of the squares of the eigenvalues of the Laplacian matrix of the graph. The latter definition was extended to digraphs in 2010 by Perera and Mizoguchi~\cite{PeMi}. In a paper of 2015~\cite{QiFL}, Qi, Fuller, Luo, Guo and Zhang obtained lower and upper bounds on the Laplacian energy of digraphs and also characterized the extremal digraphs.
Very recently, in a paper of 2020~\cite{YaWa}, Yang and Wang determined the directed trees, unicyclic digraphs and bicyclic digraphs which attain the maximal and minimal Laplacian energy among all digraphs with $n$ vertices, respectively. We refer the interested reader to the three monographs~\cite{Gut2,GuLi,LiSG} for a wide range of alternative definitions for energies of graphs and digraphs, and a wealth of references to obtained results.

Our results are closely related to and motivated by the aforementioned results obtained in~\cite{QiFL} and~\cite{YaWa}. We consider classes of digraphs with a fixed dichromatic number. This digraph invariant was introduced 40 years ago by Neumann-Lara~\cite{NeLa}. It is the smallest integer $r$ such that the digraph has a partition of its vertex set into $r$ sets, each inducing an acyclic subdigraph. As our main results, we determine the digraphs which attain the minimal and maximal Laplacian energy among all digraphs with a fixed dichromatic number, as well as for the subclass of join digraphs, which will be defined later. We also determine sharp lower and upper bounds for the third Laplacian spectral moment among all join digraphs with a fixed dichromatic number, and characterize the extremal join digraphs with dichromatic number 2. We leave the full characterization of the extremal digraphs for the $k$-th Laplacian spectral moment ($k\ge 3$) among all (join) digraphs with a fixed dichromatic number as open problems.

Before we present our results and proofs, we will recall some of the essential terminology and notation, and give some additional background and related results.

\subsection{Terminology, notation and related work}

For a digraph $G$, we use $\mathcal{V}(G)$ and $\mathcal{A}(G)$ to denote the vertex set and arc set of $G$, respectively, and we use $n=|\mathcal{V}(G)|$ and $e=|\mathcal{A}(G)|$ to denote the order and size of $G$, respectively. We denote an arc from a vertex $u$ to a vertex $v$ by $(u,v)$, and we call $u$ the tail and $v$ the head of the arc $(u,v)$. For a vertex $v\in \mathcal{V}(G)$, the outdegree $d_G^+(v)$ is the number of arcs in $\mathcal{A}(G)$ whose tail is $v$, while the indegree $d_G^-(v)$ is the number of arcs in $\mathcal{A}(G)$ whose head is $v$.
A directed walk $\pi$ of length $\ell$ from vertex $u$ to vertex $v$ in $G$ is a sequence of vertices $\pi$: $u=v_0,v_1,\ldots,v_{\ell}=v$, where $(v_{k-1},v_k)$ is an arc of $G$ for any $1\leq k\leq \ell$. If $u=v$, then $\pi$ is called a directed closed walk. If all vertices of the directed walk $\pi$ of length $\ell$ are distinct, then we call it a directed path, and denote it by $P_{\ell+1}$; a directed closed walk of length $\ell$ in which all except the end vertices are distinct is called a directed cycle, and denoted by $C_{\ell}$. We let $c_2(v)$ denote the number of directed closed walks of length $2$ associated with vertex $v$.

A key concept in the results of our research is the aforementioned extension of the notion of an acyclic digraph due to Neumann-Lara~\cite{NeLa}. A digraph $G$ is acyclic if it has no directed cycles. A vertex set $F\subseteq \mathcal{V}(G)$ is acyclic if its induced subdigraph $G[F]$ in $G$ is acyclic. A partition of $\mathcal{V}(G)$ into $r$ acyclic sets is called an $r$-coloring of $G$. Adopting the definition of~\cite{NeLa}, the minimum integer $r$ for which there exists an $r$-coloring of $G$ is the dichromatic number $\chi(G)$ of $G$. The first connection between the dichromatic number and algebraic properties related to eigenvalues of digraphs was made by Mohar in~\cite{Mo}. He extended Wilf's classical eigenvalue upper bound on the chromatic number of undirected graphs to the analogue for digraphs in terms of the dichromatic number and the spectral radius of the adjacency matrix. More recently, Drury and Lin~\cite{DrLi} determined the digraphs which have the minimum and second minimum spectral radius among all strongly connected digraphs with given order and dichromatic number. In an earlier paper, Lin and Shu~\cite{LiSh} characterized the digraphs with given dichromatic number which have the maximal spectral radius. More results about the dichromatic number of digraphs can be found in \cite{BoFJ,KiKO,LiYo,LiWC,XiWa}.

Throughout the remainder of the paper, we consider only connected digraphs without loops or multiple arcs, and we assume that
$\mathcal{V}(G)=\{v_1,v_2,\ldots,v_n\}$. We use $d_i^+$ as shorthand for $d_G^+(v_i)$ and $c^{(i)}_2$ as shorthand for $c_2(v_i)$, and we call $\left(c^{(1)}_2, c^{(2)}_2,\ldots, c^{(n)}_2\right)$ the directed closed walk sequence of length $2$ of $G$. We let
$c_2=\sum^n_{i=1}c^{(i)}_2$ denote the total number of directed closed walks of length $2$ in $G$.

The adjacency matrix $A(G)=(a_{ij})$ of $G$ is an $n\times n$ matrix whose $(i,j)$-entry equals $1$ if $(v_i, v_j)\in\mathcal{A}(G)$ and equals $0$ otherwise. The diagonal outdegree matrix $D^+(G)$ of $G$ is defined by $D^+(G)=diag(d_1^+,d_2^+,\ldots,d_n^+)$. The Laplacian matrix $L(G)$ of $G$ is defined by $L(G)=D^+(G)-A(G)$. Hence,
$L(G)=(\ell_{ij})$ is an $n \times n$ matrix, where
$$\ell_{ij}=
\begin{cases}
d_i^+,& \mbox{if} \ i=j,\\
-1,& \mbox{if} \ (v_i, v_j)\in\mathcal{A}(G),\\
0,& \mbox{otherwise}.
\end{cases}$$

For a fixed nonnegative integer $k$, the $k$-th Laplacian spectral moment of $G$ is defined as
$$LSM_k(G)=\sum_{i=1}^n \lambda_i^k,$$
where $\lambda_i$ are the eigenvalues of $L(G)$. Similarly, for the adjacency matrix, the $k$-th spectral moment of $G$ is
$$SM_k(G)=\sum_{i=1}^n z_i^k,$$
where $z_i$ are the eigenvalues of $A(G)$. It is known and straightforward to show that the latter sum equals $c_k$,  the total number of directed closed walks of length $k$ in $G$.

Some recent and very recent results involving spectral moments of undirected graphs can be found in~\cite{ChLL,FaWZ,FaZW,PaLL}.
For digraphs, results about the spectral moments of the adjacency matrix or the Laplacian matrix are generally lacking. In this paper, we are mainly concerned with the second and third Laplacian spectral moments of digraphs. Obviously, $LSM_0(G)=\sum_{i=1}^n \lambda_i^0=n$ and $LSM_1(G)=\sum_{i=1}^n \lambda_i^1=\sum_{i=1}^n d^+_i=e(G)$. For $k=2$, the second Laplacian spectral moment was first studied by Perera and Mizoguchi in~\cite{PeMi}, where they defined the Laplacian energy of a digraph $G$ as $LE(G)=\sum_{i=1}^n \lambda_i^2$. This is a direct analogue of the definition $LE(H)=\sum_{i=1}^n \lambda_i^2$ that Lazi\'c introduced in~\cite{La} for the Laplacian energy of an undirected graph $H$ with vertex set $\{v_1,v_2,\ldots,v_n\}$.  In~\cite{La}, he also proved that $LE(H)=\sum_{i=1}^n d_i(d_i+1)$, where $d_i$ is the degree of $v_i$ in $H$.

As we mentioned before, there exist many alternative definitions for graph energies of graphs and digraphs, as witnessed by the sources~\cite{Gut2,GuLi,LiSG}. This shows the popularity of this topic within chemical graph theory. Nevertheless, there are just a few published papers on the Laplacian energy of digraphs. They mainly deal with obtaining lower and upper bounds for $LE(G)$ and characterizing the extremal digraphs, as it was done in~\cite{QiFL} for general digraphs, and for the special graph classes of directed trees, unicyclic digraphs and bicyclic digraphs in~\cite{YaWa}. We extend these results to digraphs with a fixed dichromatic number.

Before we can present our results and their proofs, we need to introduce some special classes of digraphs. For a digraph $G$, the underlying graph is a graph obtained from $G$ by ignoring the direction on the arcs of $G$, i.e., by replacing each arc $(u,v)$ of $G$ with an edge joining $u$ and $v$ (possibly yielding multiple edges).

A directed tree is a digraph obtained from an undirected tree by assigning a direction to each edge, i.e., a digraph with $n$ vertices and $n-1$ arcs whose underlying graph does not contain any cycles. 
If $n=1$, then the directed tree is an isolated vertex.
An in-tree is a directed tree for which the outdegree of each vertex is at most one. Hence, an in-tree has exactly one vertex with outdegree 0, and such a vertex is called the root of the in-tree.

A tournament is a digraph obtained from an undirected complete graph by assigning a direction to each edge. A transitive tournament is a tournament $G$ satisfying the following condition: if $(u,v)\in \mathcal{A}(G)$ and $(v,w)\in \mathcal{A}(G)$, then $(u,w)\in \mathcal{A}(G)$.

Every undirected graph $H$ determines a bidirected graph $\stackrel{\leftrightarrow}{H}$ that is obtained from $H$ by replacing each edge with two oppositely arcs joining the same pair of vertices. We use $\stackrel{\leftrightarrow}{K}_n$ to denote the bidirected complete graph of order $n$, and we use $\stackrel{\leftrightarrow}{C}_n$ to denote the bidirected cycle of order $n$.

The join of two disjoint digraphs $G_1$ and $G_2$, denoted by $G_1\vee G_2$, is the digraph having vertex set $\mathcal{V}(G_1)\cup \mathcal{V}(G_2)$ and arc set $\mathcal{A}(G_1)\cup \mathcal{A}(G_2)\cup \{(u,v),(v,u)|u\in\mathcal{V}(G_1),v\in\mathcal{V}(G_2)\}$.

We use $\mathcal{G}_{n,r}$ to denote the set of digraphs of order $n$ with dichromatic number $r$. We say that a digraph with dichromatic number $r$ is a join digraph if it is the join of $r$ connected acyclic digraphs. In particular,
we let $\bigvee_{i=1}^{r} V^{i}$ denote the join digraph in $\mathcal{G}_{n,r}$ which is isomorphic to $V^1\vee V^2\vee \cdots\vee V^r$, in which each $V^i$ is a connected acyclic digraph on $n_i$ vertices, and we assume that $\sum^r_{i=1}n_i=n$ and $n_1\geq n_2\geq\cdots\geq n_r$.

The rest of the paper is organized as follows. In Section~\ref{sec2}, we obtain the digraphs which attain the minimal and maximal Laplacian energy $LE(G)$ among all (join) digraphs in $\mathcal{G}_{n,r}$. In Section~\ref{sec3}, we determine sharp upper and lower bounds for the third Laplacian spectral moment $LSM_3(G)$ among all join digraphs in $\mathcal{G}_{n,r}$. We finish the paper with some concluding remarks and open problems in Section~\ref{sec4}.

\section{Extremal digraphs for the Laplacian energy}\label{sec2}

In this section, we will characterize the digraphs which attain the minimal and maximal Laplacian energy $LE(G)$ among all join digraphs (subsection~\ref{sec2.1}) and all digraphs (subsection~\ref{sec2.2}) in $\mathcal{G}_{n,r}$. First, we list some known results and lemmas that we use in our proofs. Recall that for any digraph $G$ of order $n$ we assume that $\mathcal{V}(G)=\{v_1,v_2,\ldots, v_n\}$, that $d^+_i$ denotes the outdegree of $v_i$, and that we use $c_2$ to indicate the total number of directed closed walks of length $2$ in $G$. Also recall the assumption that we consider connected digraphs only.

\noindent\begin{lemma}\label{le:ch-2.1}(\cite{QiFL}) Let $G$ be a digraph of order $n$. Then
$$LE(G)=\sum_{i=1}^n(d^+_i)^2+c_2.$$
\end{lemma}

\noindent\begin{lemma}\label{le:ch-2.2}(\cite{PeMi,QiFL}) Let $G$ be a digraph of order $n$. Then
$$n-1 \leq LE(G) \leq n^2(n-1).$$
Moreover, the first inequality is an equality if and only if $G$ is an in-tree, and the second inequality is an equality if and only if $G$ is a bidirected complete graph.
\end{lemma}

We use the above two lemmas to prove the following counterpart of Lemma~\ref{le:ch-2.2} for acyclic digraphs. We first recall what is known as Karamata's inequality.

Let $I$ be an interval of the real line and let $f$ denote a real-valued, convex function defined on $I$. If $x_1,x_2,\ldots,x_n$ and $y_1,y_2,\ldots,y_n$ are numbers in $I$ such that $(x_1,x_2,\ldots,x_n)$ majorizes $(y_1,y_2,\ldots,y_n)$, then
$$f(x_1)+f(x_2)+\cdots+f(x_n)\geq f(y_1)+f(y_2)+\cdots+f(y_n).$$
If $f$ is a strictly convex function, then the inequality holds with equality if and only if we have $x_i=y_i$ for all $i=1,2,\ldots,n$.

Here majorization means that $x_1,x_2,\ldots,x_n$ and $y_1,y_2,\ldots,y_n$ satisfies
$$x_1\geq x_2\geq \cdots \geq x_n\ and\ y_1\geq y_2\geq \cdots \geq y_n,$$
and we have the inequalities
$$x_1+x_2+\cdots+x_i\geq y_1+y_2+\cdots+y_i,$$
for all $i=1,2,\ldots,n-1$, and the equality
$$x_1+x_2+\cdots+x_n=y_1+y_2+\cdots+y_n.$$

\noindent\begin{lemma}\label{le:ch-2.3} Let $G$ be a acyclic digraph of order $n$. Then
$$n-1 \leq LE(G) \leq \frac{n(n-1)(2n-1)}{6}.$$
Moreover, the first inequality is an equality if and only if $G$ is an in-tree, and the second inequality is an equality if and only if $G$ is a transitive tournament.
\end{lemma}
\begin{proof}
The statements about the lower bound follow directly from Lemma~\ref{le:ch-2.2}. Lemma~\ref{le:ch-2.1} implies that
$LE(G)=\sum_{i=1}^n(d_i^+)^2$ for an acyclic digraph $G$. So we need to find the maximum possible value of $\sum_{i=1}^n(d_i^+)^2$ among all acyclic digraphs. Any acyclic digraph admits a topological ordering, i.e., an ordering of its vertices
$\{v_1,v_2,\ldots, v_n\}$ such that for every arc $(v_i,v_j)$, we have $i<j$.
Using Karamata's inequality, for the acyclic digraph with order $n$ and size $e$, $\sum_{i=1}^n(d_i^+)^2$ is maximized when the outdegree sequence is $(n-1,n-2,\ldots,n-x,y,0,\ldots,0)$, where $1\leq x\leq n$, $0\leq y\leq n-x-2$ and $\frac{(n-1+n-x)x}{2}+y=e$. Clearly,
$$\sum_{i=1}^n(n-i)^2\geq \sum_{i=1}^x(n-i)^2+y^2.$$
That is, $\sum_{i=1}^n(d^+_i)^2$ is maximized when arcs $(v_i,v_j)$ exist for all $i < j$ and $j\le n$, so when $d^+_i=n-i$ for $i=1,\ldots,n$. Hence, using a well-known expression for the sum of squares, we obtain
$$LE(G)=\sum_{i=1}^n(d^+_i)^2 \leq \sum_{i=1}^n(n-i)^2=\frac{n(n-1)(2n-1)}{6}.$$

The above inequality is an equality if and only if $d^+_i=n-i$ for $i=1,\ldots,n$. It is an easy exercise and a folklore result that this is only possible if $G$ is a transitive tournament.
\end{proof}

\subsection{Extremal digraphs for the Laplacian energy among all join digraphs}\label{sec2.1}
In our first main results, we will determine the digraphs which attain the minimal and maximal Laplacian energy $LE(G)$ among all join digraphs in $\mathcal{G}_{n,r}$. We need the following inequality in our proofs of the main results in this subsection.

\noindent\begin{lemma}\label{le:ch-2.4} Let $f(x)=x^2(a-bx)$ for an integer variable $x$ and two fixed real numbers $a$ and $b$. Suppose $x_i$ and $x_j$ are chosen such that $x_i-x_j\geq2$ and $x_j<\frac{a}{3b}-1$. Then
$$f(x_i-1)+f(x_j+1)<f(x_i)+f(x_j).$$
\end{lemma}
\begin{proof}
Since
\vspace*{-0.8cm}
\begin{align*}
&f(x_i-1)+f(x_j+1)\\
&=(x_i-1)^2[a-b(x_i-1)]+(x_j+1)^2[a-b(x_j+1)]\\
&=x_i^2(a-bx_i)+bx_i^2+(1-2x_i)(a-bx_i+b)\\
&+x_j^2(a-bx_j)-bx_j^2+(1+2x_j)(a-bx_j-b),
\end{align*}
\vspace*{-0.8cm}
\indent we have
\begin{align*}
&[f(x_i-1)+f(x_j+1)]-[f(x_i)+f(x_j)]\\
&=bx_i^2+(1-2x_i)(a-bx_i+b)-bx_j^2+(1+2x_j)(a-bx_j-b)\\
&=(x_j-x_i)(-3bx_i-3bx_j+2a-3b)-6bx_i+2a\\
&\leq-2(-3bx_i-3bx_j+2a-3b)-6bx_i+2a\\
&=6bx_j-2a+6b<6b(\frac{a}{3b}-1)-2a+6b=0.
\end{align*}
\end{proof}

The next result characterizes the digraphs which attain the minimal Laplacian energy $LE(G)$ among all join digraphs $\bigvee_{i=1}^{r} V^{i}$ in $\mathcal{G}_{n,r}$.

\noindent\begin{theorem}\label{th:ch-2.5} Let $G=\bigvee_{i=1}^{r} V^{i}$. Then
$$LE(G)\geq(r-1)n^2+r^2n-r^3,$$
with equality holding if and only if each $V^i$ is an in-tree, $n_1=n-r+1$, and $n_2=\cdots=n_r=1$.
\end{theorem}
\begin{proof}
Let $\{v^i_1,v^i_2,\ldots,v^i_{n_i}\}$ be the vertex set of $V^i$, where $i=1,2,\ldots,r$. Let $d^+_G(v^i_j)$ be the outdegree of $v^i_j$ in $G$ and $d^+_{V^i}(v^i_j)$ be the outdegree of $v^i_j$ in $V^i$, where $j=1,2,\ldots,n_i$. Obviously, we have $d^+_G(v^i_j)=n-n_i+d^+_{V^i}(v^i_j)$, where $j=1,2,\ldots,n_i$ and $i=1,2,\ldots,r$. Since $V^i$ is acyclic and connected, we know that $\sum^{n_i}_{j=1}d^+_{V^i}(v^i_j)=e(V^i)\geq n_i-1$, with equality if and only if $V^i$ is a directed tree. Using Lemma~\ref{le:ch-2.3}, we also have $\sum^{n_i}_{j=1}\left(d^+_{V^i}(v^i_j)\right)^2\geq n_i-1$, with equality if and only if $V^i$ is an in-tree.

Hence, starting with the expression from Lemma~\ref{le:ch-2.1}, we have
\begin{align*}
LE(G)&=\sum_{i=1}^n(d_i^+)^2+c_2\\
&=\sum^r_{i=1}\sum^{n_i}_{j=1}\left(d^+_G(v^i_j)\right)^2+2\sum_{i<j}n_in_j\\
&=\sum^r_{i=1}\sum^{n_i}_{j=1}\left(n-n_i+d^+_{V^i}(v^i_j)\right)^2+\left[\left(\sum^r_{i=1}n_i\right)^2-\sum^r_{i=1}n_i^2\right]\\
&=\sum^r_{i=1}\left[\sum^{n_i}_{j=1}(n-n_i)^2+2(n-n_i)\sum^{n_i}_{j=1}d^+_{V^i}(v^i_j)
+\sum^{n_i}_{j=1}\left(d^+_{V^i}(v^i_j)\right)^2\right]+\left(n^2-\sum^r_{i=1}n_i^2\right)\\
&\geq \left(n^3+\sum^r_{i=1}n_i^3-2n\sum^r_{i=1}n_i^2\right)+\sum^r_{i=1}\left[2(n-n_i)(n_i-1)+(n_i-1)\right]
+\left(n^2-\sum^r_{i=1}n_i^2\right)\\
&=\left(n^3+3n^2-(2r-3)n-r\right)+\sum^r_{i=1}n_i^2(n_i-2n-3).
\end{align*}

Next, we are going to use Lemma~\ref{le:ch-2.4} to determine the minimum value of the above sum $\sum^r_{i=1}n_i^2(n_i-2n-3)$. Since $n_i-2n-3<0$, this is equivalent to determining the maximum value of $\sum^r_{i=1}n_i^2(2n+3-n_i)$.

Let $f(x)=x^2(2n+3-x)$ and $F(x_1,x_2,\ldots,x_r)=\sum^r_{i=1}f(x_i)$, where $\sum^r_{i=1}x_i=n$ and $1\leq x_i\leq n-r+1$. Suppose that $x_i-x_j\geq2$ for some $x_i$ and $x_j$. Then, since $x_i+x_j<n-(r-2)$, we have $x_j<\frac{n-r}{2}$. Now, let $a=2n+3$ and $b=1$. Since $x_j<\frac{n-r}{2}<\frac{a}{3b}-1$, using Lemma~\ref{le:ch-2.4}, we get $f(x_i-1)+f(x_j+1)<f(x_i)+f(x_j)$. Then we have $F(x_1,\ldots,x_i-1,\ldots,x_j+1,\ldots,x_r)<F(x_1,\ldots,x_i,\ldots,x_j,\ldots,x_r)$. This implies that $\sum^r_{i=1}x_i^2(2n+3-x_i)$ is maximal when $x_1=n-r+1$ and $x_2=\cdots=x_r=1$.

Concluding, we obtain
\begin{align*}
LE(G)&\geq\left(n^3+3n^2-(2r-3)n-r\right)+\sum^r_{i=1}n_i^2(n_i-2n-3)\\
&\geq\left(n^3+3n^2-(2r-3)n-r\right)\\
&+(n-r+1)^2(n-r+1-2n-3)+(r-1)(1-2n-3)\\
&=(r-1)n^2+r^2n-r^3,
\end{align*}
with equality holding here and above if and only if each $V^i$ is an in-tree, $n_1=n-r+1$, and $n_2=\cdots=n_r=1$.
\end{proof}

The next result characterizes the digraphs which attain the maximal Laplacian energy $LE(G)$ among all join digraphs $\bigvee_{i=1}^{r} V^{i}$ in $\mathcal{G}_{n,r}$.
We will distinguish the cases that $r$ is a divisor of $n$, denoted by $r\mid n$, and that $r$ does not divide $n$, denoted by $r\nmid n$.

\noindent\begin{theorem}\label{th:ch-2.6} Let $G=\bigvee_{i=1}^{r} V^{i}$. Then the following inequalities hold:

\noindent (i) If $r\mid n$, we have
$$LE(G)\leq\left(1+\frac{1}{3r^2}-\frac{1}{r}\right)n^3-\frac{1}{2r}n^2+\frac{1}{6}n,$$
with equality holding if and only if each $V^i$ is a transitive tournament with $n_i=\frac{n}{r}$.

\noindent (ii) If $r\nmid n$, we have
$$LE(G)\leq n^3+\frac{1}{6}n+p-q,$$
where $p=\left\lceil\frac{n}{r}\right\rceil^2\left(n-r\left\lfloor\frac{n}{r}\right\rfloor\right)
\left(\frac{1}{3}\left\lceil\frac{n}{r}\right\rceil-n-\frac{1}{2}\right)$ and $q=\left\lfloor\frac{n}{r}\right\rfloor^2\left(n-r\left\lceil\frac{n}{r}\right\rceil\right)
\left(\frac{1}{3}\left\lfloor\frac{n}{r}\right\rfloor-n-\frac{1}{2}\right)$.
The inequality is an equality if and only if each $V^i$ is a transitive tournament, with $n_s=\left\lceil\frac{n}{r}\right\rceil$ for $s=1,2,\ldots,n-r\left\lfloor\frac{n}{r}\right\rfloor$ and $n_t=\left\lfloor\frac{n}{r}\right\rfloor$ for $t=n-r\left\lfloor\frac{n}{r}\right\rfloor+1,n-r\left\lfloor\frac{n}{r}\right\rfloor+2,\ldots,r$.
\end{theorem}
\begin{proof}
Similarly as in the proof of Theorem~\ref{th:ch-2.5}, using Lemma~\ref{le:ch-2.3}, we have $\sum^{n_i}_{j=1}d^+_{V^i}(v^i_j)=e(V^i)\leq \frac{n_i(n_i-1)}{2}$ and $\sum^{n_i}_{j=1}\left(d^+_{V^i}(v^i_j)\right)^2\leq \frac{n_i(n_i-1)(2n_i-1)}{6}$, with equality in the latter inequality if and only if $V^i$ is a transitive tournament.

Hence, using Lemma~\ref{le:ch-2.1}, we have
\begin{align*}
LE(G)&=\sum^r_{i=1}\sum^{n_i}_{j=1}\left(d^+_G(v^i_j)\right)^2+2\sum_{i<j}n_in_j\\
&=\sum^r_{i=1}\left[\sum^{n_i}_{j=1}(n-n_i)^2+2(n-n_i)\sum^{n_i}_{j=1}d^+_{V^i}(v^i_j)
+\sum^{n_i}_{j=1}\left(d^+_{V^i}(v^i_j)\right)^2\right]+\left(n^2-\sum^r_{i=1}n_i^2\right)\\
&\leq \left(n^3+\sum^r_{i=1}n_i^3-2n\sum^r_{i=1}n_i^2\right)
+\sum^r_{i=1}\left[2(n-n_i)\frac{n_i(n_i-1)}{2}+\frac{n_i(n_i-1)(2n_i-1)}{6}\right]\\
&+\left(n^2-\sum^r_{i=1}n_i^2\right)\\
&=\left(n^3+\frac{1}{6}n\right)+\sum^r_{i=1}n_i^2\left(\frac{1}{3}n_i-n-\frac{1}{2}\right).
\end{align*}
Next, we are going to use Lemma~\ref{le:ch-2.4} to determine the maximum value of the above sum $\sum^r_{i=1}n_i^2\left(\frac{1}{3}n_i-n-\frac{1}{2}\right)$. Since $\frac{1}{3}n_i-n-\frac{1}{2}<0$, this is equivalent to determining the minimum value of $\sum^r_{i=1}n_i^2\left(n+\frac{1}{2}-\frac{1}{3}n_i\right)$.

Let $f(x)=x^2\left(n+\frac{1}{2}-\frac{1}{3}x\right)$ and $F(x_1,x_2,\ldots,x_r)=\sum^r_{i=1}f(x_i)$, where $\sum^r_{i=1}x_i=n$ and $1\leq x_i\leq n-r+1$. Let $a=n+\frac{1}{2}$ and $b=\frac{1}{3}$. Since $x_j<\frac{a}{3b}-1$, using Lemma~\ref{le:ch-2.4}, we get $f(x_i-1)+f(x_j+1)<f(x_i)+f(x_j)$ for any $x_i,x_j$ with $x_i-x_j\geq2$. Then we have $F(x_1,\ldots,x_i-1,\ldots,x_j+1,\ldots,x_r)<F(x_1,\ldots,x_i,\ldots,x_j,\ldots,x_r)$. That is, when $|x_i-x_j|\leq1$, $\sum^r_{i=1}x_i^2\left(n+\frac{1}{2}-\frac{1}{3}x_i\right)$ is minimal.

\medskip
\noindent (i) If $r\mid n$, then $|n_i-n_j|\leq1$ implies $n_i=\frac{n}{r}$ for all $i=1,2,\ldots,r$. Therefore, we obtain
\begin{align*}
LE(G)&\leq\left(n^3+\frac{1}{6}n\right)+\sum^r_{i=1}n_i^2\left(\frac{1}{3}n_i-n-\frac{1}{2}\right)\\
&\leq\left(n^3+\frac{1}{6}n\right)+r\left(\frac{n}{r}\right)^2\left(\frac{n}{3r}-n-\frac{1}{2}\right)\\
&=\left(1+\frac{1}{3r^2}-\frac{1}{r}\right)n^3-\frac{1}{2r}n^2+\frac{1}{6}n,
\end{align*}
with equality if and only if each $V^i$ is a transitive tournament of order $n_i=\frac{n}{r}$.

\noindent (ii) If $r\nmid n$, then $|n_i-n_j|\leq1$ implies $n_s=\left\lceil\frac{n}{r}\right\rceil$ for $s=1,2,\ldots,n-r\left\lfloor\frac{n}{r}\right\rfloor$ and $n_t=\left\lfloor\frac{n}{r}\right\rfloor$ for $t=n-r\left\lfloor\frac{n}{r}\right\rfloor+1,n-r\left\lfloor\frac{n}{r}\right\rfloor+2,\ldots,r$. Therefore, we obtain
\begin{align*}
LE(G)&\leq\left(n^3+\frac{1}{6}n\right)+\sum^r_{i=1}n_i^2\left(\frac{1}{3}n_i-n-\frac{1}{2}\right)\\
&\leq\left(n^3+\frac{1}{6}n\right)+\left(n-r\left\lfloor\frac{n}{r}\right\rfloor\right)\left\lceil\frac{n}{r}\right\rceil^2
\left(\frac{1}{3}\left\lceil\frac{n}{r}\right\rceil-n-\frac{1}{2}\right)\\
&+\left(r-n+r\left\lfloor\frac{n}{r}\right\rfloor\right)\left\lfloor\frac{n}{r}\right\rfloor^2
\left(\frac{1}{3}\left\lfloor\frac{n}{r}\right\rfloor-n-\frac{1}{2}\right)\\
&=\left(n^3+\frac{1}{6}n\right)+\left(n-r\left\lfloor\frac{n}{r}\right\rfloor\right)\left\lceil\frac{n}{r}\right\rceil^2
\left(\frac{1}{3}\left\lceil\frac{n}{r}\right\rceil-n-\frac{1}{2}\right)\\
&-\left(n-r\left\lceil\frac{n}{r}\right\rceil\right)\left\lfloor\frac{n}{r}\right\rfloor^2
\left(\frac{1}{3}\left\lfloor\frac{n}{r}\right\rfloor-n-\frac{1}{2}\right)\\
&=n^3+\frac{1}{6}n+p-q,
\end{align*}
where $p=\left\lceil\frac{n}{r}\right\rceil^2\left(n-r\left\lfloor\frac{n}{r}\right\rfloor\right)
\left(\frac{1}{3}\left\lceil\frac{n}{r}\right\rceil-n-\frac{1}{2}\right)$ and $q=\left\lfloor\frac{n}{r}\right\rfloor^2\left(n-r\left\lceil\frac{n}{r}\right\rceil\right)
\left(\frac{1}{3}\left\lfloor\frac{n}{r}\right\rfloor-n-\frac{1}{2}\right)$.
Obviously, the inequality is an equality  if and only if each $V^i$ is a transitive tournament, with $n_s=\left\lceil\frac{n}{r}\right\rceil$ for $s=1,2,\ldots,n-r\left\lfloor\frac{n}{r}\right\rfloor$ and $n_t=\left\lfloor\frac{n}{r}\right\rfloor$ for $t=n-r\left\lfloor\frac{n}{r}\right\rfloor+1,n-r\left\lfloor\frac{n}{r}\right\rfloor+2,\ldots,r$.

This completes the proof of Theorem~\ref{th:ch-2.6}.
\end{proof}

Let $G[n_1,n_2,\ldots,n_r]$ denote the join $\bigvee_{i=1}^{r} V^{i}$ in which each $V^i$ is either an in-tree or a transitive tournament, respectively. Then by Lemma~\ref{le:ch-2.4} and from the proof of Theorem~\ref{th:ch-2.5} (or Theorem~\ref{th:ch-2.6}, respectively), we can find a size relationship with respect to the Laplacian energies of the digraphs $G[n_1,n_2,\ldots,n_r]\in\mathcal{G}_{n,r}$ for different choices of the $n_i$. We give the following example with $n=10$ and $r=4$ to illustrate this, as shown in Figure~\ref{fi:ch-1}. Every arrow points to a digraph with a higher Laplacian energy.

\begin{figure}[htbp]
\begin{centering}
\includegraphics[scale=1.2]{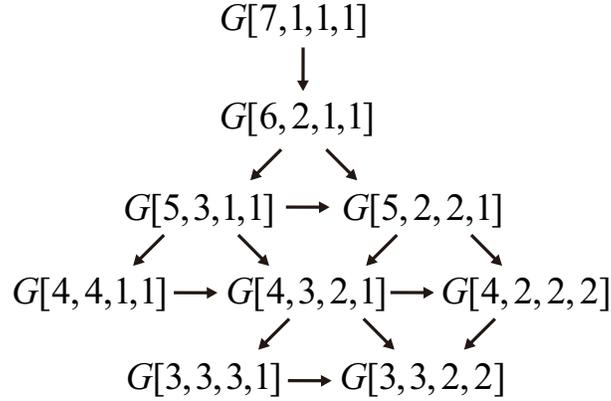}
\caption{The size relationship of the Laplacian energies of $G[n_1,n_2,n_3,n_4]\in\mathcal{G}_{10,4}$}\label{fi:ch-1}
\end{centering}
\end{figure}

\subsection{Extremal digraphs for the Laplacian energy among all digraphs}\label{sec2.2}
In our second main results, we will determine the digraphs which attain the minimal and maximal Laplacian energy $LE(G)$ among all digraphs in $\mathcal{G}_{n,r}$.

In order to characterize the digraphs which attain the minimal Laplacian energy $LE(G)$ among all digraphs in $\mathcal{G}_{n,r}$, we need some lemmas. We first list some useful results due to Mohar~\cite{Mo} involving $r$-critical digraphs. Suppose that $v\in\mathcal{V}(G)$ is a vertex such that $\chi(G-v)<\chi(G)$. Then we say that $v$ is a critical vertex. If every vertex of $G$ is critical and $\chi(G)=r$, then we say that $G$ is an $r$-critical digraph. Note that every digraph with dichromatic number at least $r$ contains an induced subdigraph that is $r$-critical.

\noindent\begin{lemma}\label{le:ch-2.7}(\cite{Mo}) If $v$ is a critical vertex in a digraph $G$ with  dichromatic number $r$, then $d^+_G(v)\geq r-1$ and $d^-_G(v)\geq r-1$.
\end{lemma}

\noindent\begin{lemma}\label{le:ch-2.8}(\cite{Mo}) Let $G$ be an $r$-critical digraph of order $n$ in which every vertex $v$ satisfies $d^+_G(v)=d^-_G(v)=r-1$. Then one of the following cases occurs:

\noindent (i) $r=2$ and $G$ is a directed cycle of length $n\geq2$.

\noindent (ii) $r=3$ and $G$ is a bidirected cycle of odd length $n\geq3$.

\noindent (iii) $G$ is a bidirected complete graph of order $r\geq4$.
\end{lemma}

We also need the following lemma.

\noindent\begin{lemma}\label{le:ch-2.9} Let $G$ be a digraph in $\mathcal{G}_{n,r}$, and let $G'$ be an $r$-critical subdigraph of $G$. If $G$ attains the minimal Laplacian energy $LE(G)$ among all digraphs in $\mathcal{G}_{n,r}$, then $d^+_G(v)=1$ for any $v\in \mathcal{V}(G)\setminus\mathcal{V}(G')$ and $d^+_G(u)=d^+_{G'}(u)$ for any $u\in\mathcal{V}(G')$.
\end{lemma}
\begin{proof}
Suppose that $G$ attains the minimal Laplacian energy $LE(G)$ among all digraphs in $\mathcal{G}_{n,r}$. First, we prove $d^+_G(v)=1$ for any $v\in \mathcal{V}(G)\setminus\mathcal{V}(G')$. We start with the following claim.

\noindent\begin{claim}\label{cl:ch-1}
$d^+_G(v)\neq0$ for any $v\in \mathcal{V}(G)\setminus\mathcal{V}(G')$.
\end{claim}

\noindent Suppose there exists a vertex $v\in \mathcal{V}(G)\setminus\mathcal{V}(G')$ such that $d^+_G(v)=0$. Then $d^-_G(v)\geq1$, since $G$ is connected. Let $(w,v)\in\mathcal{A}(G)$ for $w\in\mathcal{V}(G)$. Let $G^1$ be obtained from $G$ by reversing the direction on the arc $(w,v)$, denoted as
$$G^1=G-(w,v)+(v,w).$$
Then
$$LE(G^1)=LE(G)-\left(d^+_G(w)\right)^2+\left(d^+_G(w)-1\right)^2+1.$$
We discuss the possible choices for $w$, and derive contradictions in all of the three cases.

\medskip\noindent
\textbf{Case 1:} $w\in\mathcal{V}(G')$. Since $G'$ is $r$-critical, $d^+_{G'}(w)\geq1$. So we have $d^+_G(w)>1$. Then $LE(G^1)<LE(G)$, a contradiction to $LE(G)$ being minimal.

\medskip\noindent
\textbf{Case 2:} $w\in \mathcal{V}(G)\setminus\mathcal{V}(G')$ and $d^+_G(w)>1$. Obviously $LE(G^1)<LE(G)$, a contradiction.

\medskip\noindent
\textbf{Case 3:} $w\in \mathcal{V}(G)\setminus\mathcal{V}(G')$ and $d^+_G(w)=1$. Then $LE(G^1)=LE(G)$ and we know $d^+_{G^1}(w)=0$. So, for $G^1$ there also exists a vertex $v\in \mathcal{V}(G^1)\setminus\mathcal{V}(G')$ such that $d^+_{G^1}(v)=0$. We use the following procedure:

\medskip
$H^0:=G$;

$i:=0$;

\textbf{while} $\exists\ v\in\mathcal{V}(H^i)\setminus\mathcal{V}(G')$ s.t. $d^+_{H^i}(v)=0$ \textbf{do begin}

\ \ \ \ \ \  select a vertex $w\in\mathcal{V}(H^i)$ with $(w,v)\in\mathcal{A}(H^i)$;

\ \ \ \ \ \ $H^{i+1}:=H^i-(w,v)+(v,w)$;

\ \ \ \ \ \ $i:=i+1$;

\textbf{end.}
\medskip

The resulting digraph $H$ we obtain at the termination of this procedure has no vertex $v\in\mathcal{V}(H)\setminus\mathcal{V}(G')$ such that $d^+_H(v)=0$. Hence, as above we conclude that $LE(H)<LE(G)$, a contradiction. This completes the proof of Claim~\ref{cl:ch-1}.

We also need the following claim.

\noindent\begin{claim}\label{cl:ch-2}
Every component of $G-\mathcal{V}(G')$ is an in-tree, the root of which is an inneighbor of exactly one vertex of $G'$.
\end{claim}

\noindent
Let $T$ be a component of $G-\mathcal{V}(G')$. Then Lemma~\ref{le:ch-2.2} implies that $T$ is a in-tree. Let $v_0\in\mathcal{V}(T)$ be the root of $T$. From Claim~\ref{cl:ch-1}, we have $d^+_G(v_0)\neq0$. Since $G$ is connected, by the minimality of $LE(G)$, $d^+_G(v_0)=1$. That is, the root of the in-tree $T$ is the inneighbor of exactly one vertex of $G'$. This completes the proof of Claim~\ref{cl:ch-2}.

From Claims~\ref{cl:ch-1} and~\ref{cl:ch-2}, we get that $d^+_G(v)=1$ for any $v\in\mathcal{V}(G)\setminus\mathcal{V}(G')$. Next we will prove that $d^+_G(u)=d^+_{G'}(u)$ for any $u\in\mathcal{V}(G')$.

Suppose there exists a vertex $u\in\mathcal{V}(G')$ such that $d^+_G(u)>d^+_{G'}(u)$. Then there exists an arc $(u,v)\in\mathcal{A}(G)$ for $v\in\mathcal{V}(G)\setminus\mathcal{V}(G')$. Let
$$G^2=G-(u,v).$$
Claim~\ref{cl:ch-2} implies that $G^2$ is connected. Clearly, $LE(G^2)<LE(G)$, a contradiction. Hence, we conclude that $d^+_G(u)=d^+_{G'}(u)$ for any $u\in\mathcal{V}(G')$.
\end{proof}

Now, we are ready to present and prove the main result of this section. It characterizes the digraphs which attain the minimal Laplacian energy $LE(G)$ among all digraphs in $\mathcal{G}_{n,r}$.

\noindent\begin{theorem}\label{th:ch-2.10} Let $G$ be a digraph in $\mathcal{G}_{n,r}$. Then the following inequalities hold:

\noindent (i) If $r=2$, we have
$$LE(G)\geq \begin{cases}
4,& \mbox{if} \ n=2,\\
n,& \mbox{if} \ n\geq3.
\end{cases}$$
If $n=2$, the inequality is an equality if and only if $G$ is a directed cycle $C_2$. If $n\geq3$, the inequality is an equality if and only if $G$ contains a directed cycle $C_{n'}$ $(n'\geq3)$ and every component (if any) of $G-\mathcal{V}(C_{n'})$ is an in-tree, the root of which is an inneighbor of exactly one vertex of $C_{n'}$.

\noindent (ii) If $r\geq3$, we have
$$LE(G)\geq n+r^3-r^2-r,$$
with equality holding if and only if $G$ contains a bidirected complete graph $\stackrel{\leftrightarrow}{K}_r$ and every component (if any) of $G-\mathcal{V}(\stackrel{\leftrightarrow}{K}_r)$ is an in-tree, the root of which is an inneighbor of exactly one vertex of $\stackrel{\leftrightarrow}{K}_r$.
\end{theorem}

\begin{proof}
Let $G$ be a digraph in $\mathcal{G}_{n,r}$. Then $G$ must contain an induced subdigraph $G'$ of order $n'$ that is $r$-critical. From Lemma~\ref{le:ch-2.9}, we obtain that if $G$ attains the minimal Laplacian energy $LE(G)$ among all digraphs in $\mathcal{G}_{n,r}$, then $d^+_G(v)=1$ for any $v\in\mathcal{V}(G)\setminus\mathcal{V}(G')$ and $d^+_G(u)=d^+_{G'}(u)$ for any $u\in\mathcal{V}(G')$. That is, $G$ contains an $r$-critical digraph $G'$ and every component of $G-\mathcal{V}(G')$ is an in-tree, the root of which is an inneighbor of exactly one vertex of $G'$. So, it suffices to characterize the $r$-critical digraphs which attain the minimal Laplacian energy.

By definition, every vertex of $G'$ is critical. Hence, using Lemma~\ref{le:ch-2.7}, we know that each vertex $u\in\mathcal{V}(G')$ satisfies $d^+_{G'}(u)\geq r-1$ and $d^-_{G'}(u)\geq r-1$. Recalling that $LE(G')=\sum_{u\in\mathcal{V}(G')} \left(d^+_{G'}(u)\right)^2+c_2(G')$, and noting that $c_2(u)\le d^+_{G'}(u)$, straightforward calculations show that $LE(G')$ is minimal if $d^+_{G'}(u)=r-1$ for each vertex $u\in\mathcal{V}(G')$. This means we can apply the characterizations of Lemma~\ref{le:ch-2.8}, implying the following conclusions.

If $r=2$, then $LE(G')$ is minimal when $G'$ is a directed cycle of length $n'\geq2$. So (i) follows.

If $r=3$, then $LE(G')$ is minimal when $G'$ is a bidirected cycle of odd length $n'\geq3$. In that case,
$$LE(G)\geq LE(G')+n-n'=4n'+n'(n'-1)+n-n'\geq n+15,$$
with equality holding if and only if $G$ contains a bidirected complete graph $\stackrel{\leftrightarrow}{K}_3$ (which is a bidirected cycle $\stackrel{\leftrightarrow}{C}_3$) and every component of $G-\mathcal{V}(\stackrel{\leftrightarrow}{K}_3)$ is an in-tree, the root of which is an inneighbor of exactly one vertex of $\stackrel{\leftrightarrow}{K}_3$.

If $r\geq4$, then $LE(G')$ is minimal when $G'$ is a bidirected complete graph of order $r\geq4$. In that case,
$$LE(G)\geq LE(G')+n-n'=r(r-1)^2+r(r-1)+n-r=n+r^3-r^2-r,$$
with equality holding if and only if $G$ contains a bidirected complete graph $\stackrel{\leftrightarrow}{K}_r$ and every component of $G-\mathcal{V}(\stackrel{\leftrightarrow}{K}_r)$ is an in-tree, the root of which is an inneighbor of exactly one vertex of $\stackrel{\leftrightarrow}{K}_r$.

This completes the proof of Theorem~\ref{th:ch-2.10}.
\end{proof}

As an illustration of the above theorem, Figure~\ref{fi:ch-2} shows all digraphs in $\mathcal{G}_{6,3}$ attaining the minimal Laplacian energy.

\begin{figure}[htbp]
\begin{centering}
\includegraphics[scale=1.2]{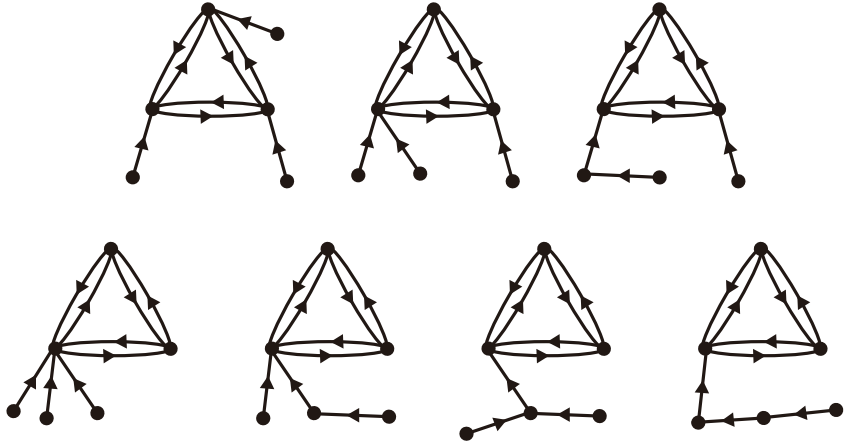}
\caption{The digraphs in $\mathcal{G}_{6,3}$ with the minimal Laplacian energy}\label{fi:ch-2}
\end{centering}
\end{figure}

The next result characterizes the digraphs which attain the maximal Laplacian energy $LE(G)$ among all digraphs in $\mathcal{G}_{n,r}$. It is an easy consequence of Theorem~\ref{th:ch-2.6}, which we state without proof.

\newpage
\noindent\begin{theorem}\label{th:ch-2.11} Let $G$ be a digraph in $\mathcal{G}_{n,r}$. Then the following inequalities hold:

\noindent (i) If $r\mid n$, we have
$$LE(G)\leq\left(1+\frac{1}{3r^2}-\frac{1}{r}\right)n^3-\frac{1}{2r}n^2+\frac{1}{6}n,$$
with equality holding if and only if $G=\bigvee_{i=1}^{r} V^{i}$ and each $V^i$ is a transitive tournament with $n_i=\frac{n}{r}$.

\noindent (ii) If $r\nmid n$, we have
$$LE(G)\leq n^3+\frac{1}{6}n+p-q,$$
where $p=\left\lceil\frac{n}{r}\right\rceil^2\left(n-r\left\lfloor\frac{n}{r}\right\rfloor\right)
\left(\frac{1}{3}\left\lceil\frac{n}{r}\right\rceil-n-\frac{1}{2}\right)$ and $q=\left\lfloor\frac{n}{r}\right\rfloor^2\left(n-r\left\lceil\frac{n}{r}\right\rceil\right)
\left(\frac{1}{3}\left\lfloor\frac{n}{r}\right\rfloor-n-\frac{1}{2}\right)$.
The inequality is an equality if and only if $G=\bigvee_{i=1}^{r} V^{i}$ and each $V^i$ is a transitive tournament, with $n_s=\left\lceil\frac{n}{r}\right\rceil$ for $s=1,2,\ldots,n-r\left\lfloor\frac{n}{r}\right\rfloor$ and $n_t=\left\lfloor\frac{n}{r}\right\rfloor$ for $t=n-r\left\lfloor\frac{n}{r}\right\rfloor+1,n-r\left\lfloor\frac{n}{r}\right\rfloor+2,\ldots,r$.
\end{theorem}

\section{Bounds for the third Laplacian spectral moment}\label{sec3}
In this section, we will determine the sharp bounds for the third Laplacian spectral moment $LSM_3(G)$ of join digraphs in $\mathcal{G}_{n,r}$. First, we present a general formula for $LSM_3(G)$ of a digraph $G$. Recall that we assume $G$ has vertex set $\{v_1,v_2,\ldots,v_n\}$, with outdegrees $d^+_1, d^+_2, \ldots , d^+_n$, and that $\left(c^{(1)}_2, c^{(2)}_2,\ldots, c^{(n)}_2\right)$ denotes the directed closed walk sequence of length $2$. We let $c_3$ denote the total number of directed closed walks of length $3$ in $G$.

\noindent\begin{lemma}\label{le:ch-3.1} Let $G$ be a digraph of order $n$. Then
$$LSM_3(G)=\sum_{i=1}^n(d^+_i)^3+3\sum_{i=1}^nd_i^+c^{(i)}_2-c_3.$$
\end{lemma}
\begin{proof}
Since the Laplacian matrix is an $n \times n$ matrix $L(G)=(\ell_{ij})$, where
$$\ell_{ij}=
\begin{cases}
d_i^+,& \mbox{if} \ i=j,\\
-1,& \mbox{if} \ (v_i, v_j)\in\mathcal{A}(G),\\
0,& \mbox{otherwise},
\end{cases}$$
we have
$$LSM_3(G)=\sum_{i=1}^n(\lambda^+_i)^3=tr\left((L(G))^3\right)=\sum_{j_1=1}^n\sum_{j_2=1}^n\sum_{j_3=1}^n \ell_{j_1j_2}\ell_{j_2j_3}\ell_{j_3j_1}.$$
For the different (possible) choices of $j_1$, $j_2$, $j_3$, we presented the respective values in Table~\ref{ta:ch-1}.

\begin{table}[ht]
\renewcommand\arraystretch{1.5}\centering
\caption{The values of $\sum_{j_1=1}^n\sum_{j_2=1}^n\sum_{j_3=1}^n\ell_{j_1j_2}\ell_{j_2j_3}\ell_{j_3j_1}$ for different choices
of $j_1$, $j_2$, $j_3$}\label{ta:ch-1}
\begin{tabular}{|c|c|c|c|}
\hline
\multicolumn{3}{|c|}{$j_1$, $j_2$, $j_3$} & $\sum_{j_1=1}^n\sum_{j_2=1}^n\sum_{j_3=1}^n\ell_{j_1j_2}\ell_{j_2j_3}\ell_{j_3j_1}$  \\ \hline
\multirow{4}*{$j_1=j_2$}
& \multirow{2}*{$j_2=j_3$}
& $j_3=j_1$ & $\sum_{i=1}^n(d^+_i)^3$  \\ \cline{3-4}
&& $j_3\neq j_1$ & non-existent  \\ \cline{2-4}
& \multirow{2}*{$j_2\neq j_3$}
& $j_3=j_1$ & non-existent  \\ \cline{3-4}
&& $j_3\neq j_1$ & $\sum_{i=1}^nd_i^+c^{(i)}_2$  \\ \cline{1-4}
\multirow{4}*{$j_1\neq j_2$}
& \multirow{2}*{$j_2=j_3$}
& $j_3=j_1$ & non-existent  \\ \cline{3-4}
&& $j_3\neq j_1$ & $\sum_{i=1}^nd_i^+c^{(i)}_2$  \\ \cline{2-4}
& \multirow{2}*{$j_2\neq j_3$}
& $j_3=j_1$ & $\sum_{i=1}^nd_i^+c^{(i)}_2$  \\ \cline{3-4}
&& $j_3\neq j_1$ & $-c_3$  \\ \cline{2-4}
\hline
\end{tabular}
\end{table}

So, we get $LSM_3(G)=\sum_{i=1}^n(d^+_i)^3+3\sum_{i=1}^nd_i^+c^{(i)}_2-c_3$.
\end{proof}

We will use the above lemma to obtain an expression for $LSM_3(G)$ in case $G=\bigvee_{i=1}^{r} V^{i}$. We adopt the notation of the previous section.

\newpage
\noindent\begin{lemma}\label{le:ch-3.2} Let $G=\bigvee_{i=1}^{r} V^{i}$. Then
$$LSM_3(G)=\sum_{i=1}^r\sum_{j=1}^{n_i}\left(d^+_G(v^i_j)\right)^3+3\sum_{i=1}^rn_i(n-n_i)^2-\sum_{i=1}^rn_i\sum_{s\neq i}n_s(n-n_s-n_i).$$
\end{lemma}
\begin{proof}
From Lemma~\ref{le:ch-3.1}, we obtain
$$LSM_3(G)=\sum_{i=1}^r\sum_{j=1}^{n_i}\left(d^+_G(v^i_j)\right)^3+3\sum_{i=1}^r\sum_{j=1}^{n_i} d^+_G(v^i_j)c_2(v^i_j)-c_3.$$
We also recall that $d^+_G(v^i_j)=n-n_i+d^+_{V^i}(v^i_j)$ and $c_2(v^i_j)=n-n_i$, for $j=1,2,\ldots,n_i$ and $i=1,2,\ldots,r$. So, we next consider $c_3$.

Let $c_3(v^i_j)$ be the number of directed closed walks of length $3$ associated with $v^i_j$. Then $c_3=\sum_{i=1}^r\sum_{j=1}^{n_i}c_3(v^i_j)$ is the total number of directed closed walks of length $3$ in $G$. Actually, any directed closed walk of length $3$ associated with $v^i_j$ is a triangle that starts and ends in $v^i_j$. For any $v^i_j\in \mathcal{V}(V^i)$, we denote the associated triangles by $v^i_j\rightarrow u\rightarrow w\rightarrow v^i_j$ and discuss the possible choices for $u$ and $w$, and their contribution to the total number of triangles.

\medskip\noindent
\textbf{Case 1:}
$u\in \mathcal{V}(V^i)$. The total contribution is clearly $d^+_{V^i}(v^i_j)(n-n_i)$.

\medskip\noindent
\textbf{Case 2:}
$u\notin \mathcal{V}(V^i)$. Let $u\in \mathcal{V}(V^s)$ for $s\neq i$. Next, we consider $w$.

\medskip\noindent
\textbf{Case 2.1:}
$w\in \mathcal{V}(V^s)$. Then the total contribution is $\sum_{t=1}^{n_s}d^+_{V^s}(v^s_t)$.

\medskip\noindent
\textbf{Case 2.2:}
$w\notin \mathcal{V}(V^s)$. Then the total contribution is $n_s\left[(n-n_s-n_i)+d^-_{V^i}(v^i_j)\right]$.

\medskip\noindent
Hence, in Case 2, we get a total contribution of
$$\sum_{s\neq i}\left[\sum_{t=1}^{n_s}d^+_{V^s}(v^s_t)+n_s\left[(n-n_s-n_i)+d^-_{V^i}(v^i_j)\right]\right].$$

Summing up, we get
\begin{align*}
c_3&=\sum_{i=1}^r\sum_{j=1}^{n_i}c_3(v^i_j)\\
&=\sum_{i=1}^r\sum_{j=1}^{n_i}\left[d^+_{V^i}(v^i_j)(n-n_i)+\sum_{s\neq i}\left[\sum_{t=1}^{n_s}d^+_{V^s}(v^s_t)+n_s\left[(n-n_s-n_i)+d^-_{V^i}(v^i_j)\right]\right]\right].
\end{align*}

Thus,
\begin{align*}
LSM_3(G)&=\sum_{i=1}^r\sum_{j=1}^{n_i}\left(d^+_G(v^i_j)\right)^3+3\sum_{i=1}^r\sum_{j=1}^{n_i} d^+_G(v^i_j)c_2(v^i_j)-c_3\\
&=\sum_{i=1}^r\sum_{j=1}^{n_i}\left(d^+_G(v^i_j)\right)^3+3\sum_{i=1}^r\sum_{j=1}^{n_i} \left(n-n_i+d^+_{V^i}(v^i_j)\right)(n-n_i)\\
&-\sum_{i=1}^r\sum_{j=1}^{n_i}\left[d^+_{V^i}(v^i_j)(n-n_i)+\sum_{s\neq i}\left[\sum_{t=1}^{n_s}d^+_{V^s}(v^s_t)+n_s\left[(n-n_s-n_i)+d^-_{V^i}(v^i_j)\right]\right]\right]\\
&=\sum_{i=1}^r\sum_{j=1}^{n_i}\left(d^+_G(v^i_j)\right)^3+3\sum_{i=1}^r\left[n_i(n-n_i)^2+(n-n_i)e(V^i)\right]\\
&-\sum_{i=1}^r(n-n_i)e(V^i)-\sum_{i=1}^r\left[n_i\sum_{s\neq i}e(V^s)+n_i\sum_{s\neq i}n_s(n-n_s-n_i)+e(V^i)\sum_{s\neq i}n_s\right]\\
&=\sum_{i=1}^r\sum_{j=1}^{n_i}\left(d^+_G(v^i_j)\right)^3+3\sum_{i=1}^rn_i(n-n_i)^2-\sum_{i=1}^rn_i\sum_{s\neq i}n_s(n-n_s-n_i)\\
&+\sum_{i=1}^r e(V^i)\left[2(n-n_i)-\sum_{s\neq i}n_s\right]-\sum_{i=1}^rn_i\sum_{s\neq i}e(V^s).
\end{align*}

Since
\begin{align*}
&\sum_{i=1}^r e(V^i)\left[2(n-n_i)-\sum_{s\neq i}n_s\right]-\sum_{i=1}^rn_i\sum_{s\neq i}e(V^s)\\
&=\sum_{i=1}^r e(V^i)(n-n_i)-\sum_{i=1}^rn_i\sum_{s\neq i}e(V^s)\\
&=n\sum_{i=1}^r e(V^i)-\sum_{i=1}^r n_i\left[e(V^i)+\sum_{s\neq i}e(V^s)\right]\\
&=n\sum_{i=1}^r e(V^i)-\sum_{i=1}^r n_i\left(\sum^r_{i=1}e(V^i)\right)\\
&=0,
\end{align*}
we obtain
$$LSM_3(G)=\sum_{i=1}^r\sum_{j=1}^{n_i}\left(d^+_G(v^i_j)\right)^3+3\sum_{i=1}^rn_i(n-n_i)^2-\sum_{i=1}^rn_i\sum_{s\neq i}n_s(n-n_s-n_i).$$
This completes the proof of the lemma.
\end{proof}

Next, using the above expression we will determine sharp bounds for $LSM_3(G)$ of join digraphs in $\mathcal{G}_{n,r}$.

\noindent\begin{theorem}\label{th:ch-3.3} Let $G=\bigvee_{i=1}^{r} V^{i}$. Then

\medskip
\noindent (i)
\vspace*{-0.8cm}
\begin{align*}
LSM_3(G)&\geq\sum_{i=1}^r\left[-n_i^4+(3n+6)n_i^3-(3n^2+12n+6)n_i^2+(n^3+6n^2+9n+4)n_i\right]\\
&-r(3n^2+3n+1)-\sum_{i=1}^rn_i\sum_{s\neq i}n_s(n-n_s-n_i),
\end{align*}
with equality holding if and only if each $V^i$ is an in-tree.

\medskip
\noindent (ii)
\vspace*{-0.8cm}
\begin{align*}
LSM_3(G)&\leq\sum_{i=1}^r\left[-\frac{1}{4}n_i^4+\left(n+\frac{5}{2}\right)n_i^3
-\left(\frac{3n^2}{2}+\frac{9n}{2}+\frac{1}{4}\right)n_i^2+\left(n^3+\frac{3n^2}{2}+\frac{n}{2}\right)n_i\right]\\
&-\sum_{i=1}^rn_i\sum_{s\neq i}n_s(n-n_s-n_i),
\end{align*}
with equality holding if and only if each $V^i$ is a transitive tournament.
\end{theorem}
\begin{proof}
We start with the expression from Lemma~\ref{le:ch-3.2}.
$$LSM_3(G)=\sum_{i=1}^r\sum_{j=1}^{n_i}\left(d^+_G(v^i_j)\right)^3+3\sum_{i=1}^rn_i(n-n_i)^2-\sum_{i=1}^rn_i\sum_{s\neq i}n_s(n-n_s-n_i).$$

Since $d^+_G(v^i_j)=n-n_i+d^+_{V^i}(v^i_j)$, for the first summation on the right-hand side, we obtain the following equality. $\sum_{i=1}^r\sum_{j=1}^{n_i}\left(d^+_G(v^i_j)\right)^3=$
\begin{align*}
\sum_{i=1}^r\left[n_i(n-n_i)^3+\sum_{j=1}^{n_i}\left(d^+_{V^i}(v^i_j)\right)^3
+3(n-n_i)\sum_{j=1}^{n_i}\left(d^+_{V^i}(v^i_j)\right)^2+3(n-n_i)^2\sum_{j=1}^{n_i}d^+_{V^i}(v^i_j)\right].
\end{align*}

Using Lemma~\ref{le:ch-2.3}, its proof, and similar calculations, we obtain the following lower and upper bounds which we will use to simplify some of the terms on the right-hand side of the above expression.
$$n_i-1\leq\sum_{j=1}^{n_i}d^+_{V^i}(v^i_j)\leq\frac{n_i(n_i-1)}{2},$$
$$n_i-1\leq\sum_{j=1}^{n_i}\left(d^+_{V^i}(v^i_j)\right)^2\leq\frac{n_i(n_i-1)(2n_i-1)}{6},$$
$$n_i-1\leq\sum_{j=1}^{n_i}\left(d^+_{V^i}(v^i_j)\right)^3\leq\frac{n_i^2(n_i-1)^2}{4}.$$
In all of the above three inequalities, the lower bounds are only attained if $V^i$ is an in-tree, and the upper bounds are only attained if $V^i$ is a transitive tournament.

Combining the above terms, for the lower bound on $LSM_3(G)$ we obtain $LSM_3(G)$
\begin{align*}
&=\sum_{i=1}^r\left[n_i(n-n_i)^3+\sum_{j=1}^{n_i}\left(d^+_{V^i}(v^i_j)\right)^3
+3(n-n_i)\sum_{j=1}^{n_i}\left(d^+_{V^i}(v^i_j)\right)^2+3(n-n_i)^2\sum_{j=1}^{n_i}d^+_{V^i}(v^i_j)\right]\\
&+3\sum_{i=1}^rn_i(n-n_i)^2-\sum_{i=1}^rn_i\sum_{s\neq i}n_s(n-n_s-n_i)\\
&\geq\sum_{i=1}^r\left[n_i(n-n_i)^3+(n_i-1)+3(n-n_i)(n_i-1)+3(n-n_i)^2(n_i-1)+3n_i(n-n_i)^2\right]\\
&-\sum_{i=1}^rn_i\sum_{s\neq i}n_s(n-n_s-n_i)\\
&=\sum_{i=1}^r\left[-n_i^4+(3n+6)n_i^3-(3n^2+12n+6)n_i^2+(n^3+6n^2+9n+4)n_i\right]-r(3n^2+3n+1)\\
&-\sum_{i=1}^rn_i\sum_{s\neq i}n_s(n-n_s-n_i),
\end{align*}
with equality holding if and only if each $V^i$ is an in-tree with $n_i$ vertices.

For the upper bound we obtain $LSM_3(G)$
\begin{align*}
&=\sum_{i=1}^r\left[n_i(n-n_i)^3+\sum_{j=1}^{n_i}\left(d^+_{V^i}(v^i_j)\right)^3
+3(n-n_i)\sum_{j=1}^{n_i}\left(d^+_{V^i}(v^i_j)\right)^2+3(n-n_i)^2\sum_{j=1}^{n_i}d^+_{V^i}(v^i_j)\right]\\
&+3\sum_{i=1}^rn_i(n-n_i)^2-\sum_{i=1}^rn_i\sum_{s\neq i}n_s(n-n_s-n_i)\\
&\leq\sum_{i=1}^r\left[n_i(n-n_i)^3+\frac{(n_i-1)^2n_i^2}{4}+3(n-n_i)\frac{n_i(n_i-1)(2n_i-1)}{6}
+3(n-n_i)^2\frac{n_i(n_i-1)}{2}\right]\\
&+\sum_{i=1}^r3n_i(n-n_i)^2-\sum_{i=1}^rn_i\sum_{s\neq i}n_s(n-n_s-n_i)\\
&=\sum_{i=1}^r\left[-\frac{1}{4}n_i^4+\left(n+\frac{5}{2}\right)n_i^3
-\left(\frac{3n^2}{2}+\frac{9n}{2}+\frac{1}{4}\right)n_i^2+\left(n^3+\frac{3n^2}{2}+\frac{n}{2}\right)n_i\right]\\
&-\sum_{i=1}^rn_i\sum_{s\neq i}n_s(n-n_s-n_i),
\end{align*}
with equality holding if and only if each $V^i$ is a transitive tournament with $n_i$ vertices.
\end{proof}

Unfortunately, we are not able to characterize the extremal digraphs for $LSM_3(G)$ precisely, as counterparts of Theorem~\ref{th:ch-2.5} and Theorem~\ref{th:ch-2.6}, except for the case when $r=2$. Based on the expressions in Theorem~\ref{th:ch-3.3}, we next determine the digraphs which attain the minimal and maximal value of
$LSM_3(G)$ among the join digraphs in $\mathcal{G}_{n,2}$.

\newpage
\noindent\begin{corollary}\label{co:ch-3.4} Let $G=V^1\vee V^2$. Then

\medskip\noindent
(i) $LSM_3(G)\geq n^3+8n-16,$ \\
with equality holding if and only if $V^1$ and $V^2$ are in-trees with $n_1=n-1$ and $n_2=1$.

\medskip\noindent
(ii) $LSM_3(G)\leq
\begin{cases}
\frac{1}{32}(15n^4-4n^3+12n^2),& \mbox{if $n$ is even},\\
\frac{1}{32}(15n^4-4n^3+6n^2-12n-5),& \mbox{if $n$ is odd},
\end{cases}$ \\
with equality holding if and only if $V^1$ and $V^2$ are transitive tournaments with $n_1=\left\lceil\frac{n}{2}\right\rceil$ and $n_2=\left\lfloor\frac{n}{2}\right\rfloor$.
\end{corollary}
\begin{proof} \text{ }

\noindent
(i) Since $n_1\geq n_2$, we have $\lceil\frac{n}{2}\rceil\leq n_1\leq n-1$ and $1\leq n_2\leq\lfloor\frac{n}{2}\rfloor$. Let $n_2=x$ and $n_1=n-x$. Using Theorem~\ref{th:ch-3.3}, we have
\begin{align*}
LSM_3(G)&\geq -\left(n_1^4+n_2^4\right)+(3n+6)\left(n_1^3+n_2^3\right)-(3n^2+12n+6)\left(n_1^2+n_2^2\right)\\
&+(n^3+6n^2+9n+4)n-2(3n^2+3n+1)\\
&=-\left((n-x)^4+x^4\right)+(3n+6)\left((n-x)^3+x^3\right)-(3n^2+12n+6)\left((n-x)^2+x^2\right)\\
&+(n^3+6n^2+9n+4)n-2(3n^2+3n+1)\\
&=-2x^4+4nx^3-(3n^2+6n+12)x^2+(n^3+6n^2+12n)x-(3n^2+2n+2).
\end{align*}

Let $f(x)=-2x^4+4nx^3-(3n^2+6n+12)x^2+(n^3+6n^2+12n)x-(3n^2+2n+2)$. Next, we prove that $f(x)$ is an increasing function when $1\leq x\leq\frac{n}{2}$.

Since $f'(x)=-8x^3+12nx^2-(6n^2+12n+24)x+(n^3+6n^2+12n)$ and $f''(x)=-24x^2+24nx-(6n^2+12n+24)$, we get $f''(x)<0$ when $1\leq x\leq\frac{n}{2}$. So, $f'(x)$ is a decreasing function and $f'(x)\geq f'(x)_{min}=f'(\frac{n}{2})=0$. Hence, $f(x)$ is an increasing function when $1\leq x\leq\frac{n}{2}$, and consequently $f(x)\geq f(1)=n^3+8n-16$.

Concluding, we obtain $LSM_3(G)\geq n^3+8n-16,$ with equality holding if and only if $V^1$ and $V^2$ are in-trees with $n_1=n-1$ and $n_2=1$.

\medskip\noindent
(ii) Similarly as in the proof of (i), using Theorem~\ref{th:ch-3.3}, we have
\begin{align*}
LSM_3(G) &\leq -\frac{1}{4}\left(n_1^4+n_2^4\right)+\left(n+\frac{5}{2}\right)\left(n_1^3+n_2^3\right)
-\left(\frac{3n^2}{2}+\frac{9n}{2}+\frac{1}{4}\right)\left(n_1^2+n_2^2\right)\\
&+\left(n^3+\frac{3n^2}{2}+\frac{n}{2}\right)n\\
&=-\frac{1}{4}\left((n-x)^4+x^4\right)+\left(n+\frac{5}{2}\right)\left((n-x)^3+x^3\right)\\
&-\left(\frac{3n^2}{2}+\frac{9n}{2}+\frac{1}{4}\right)\left((n-x)^2+x^2\right)+\left(n^3+\frac{3n^2}{2}+\frac{n}{2}\right)n\\
&=-\frac{1}{2}x^4+nx^3-\frac{1}{2}(3n^2+3n+1)x^2+\frac{1}{2}(2n^3+3n^2+n)x+\frac{1}{4}(n^4-2n^3+n^2).
\end{align*}

Let $g(x)=-\frac{1}{2}x^4+nx^3-\frac{1}{2}(3n^2+3n+1)x^2+\frac{1}{2}(2n^3+3n^2+n)x+\frac{1}{4}(n^4-2n^3+n^2)$. Next, we prove that $g(x)$ is an increasing function when $1\leq x\leq\frac{n}{2}$.

Since $g'(x)=-2x^3+3nx^2-(3n^2+3n+1)x+\frac{1}{2}(2n^3+3n^2+n)$ and $g''(x)=-6x^2+6nx-(3n^2+3n+1)$, we get $g''(x)<0$ when $1\leq x\leq\frac{n}{2}$. So, $g'(x)$ is a decreasing function and $g'(x)\geq g'(x)_{min}=g'(\frac{n}{2})=0$. Hence, $g(x)$ is an increasing function when $1\leq x\leq\frac{n}{2}$, and consequently $g(x)\leq g(\frac{n}{2})$.

Substituting $x=\frac{n}{2}$ if $n$ is even, and $x=\lfloor\frac{n}{2}\rfloor=\frac{n-1}{2}$ if $n$ is odd, we obtain $LSM_3(G)\leq\frac{1}{32}(15n^4-4n^3+12n^2)$ for even $n$, and $LSM_3(G)\leq\frac{1}{32}(15n^4-4n^3+6n^2-12n-5)$ for odd $n$, with equality holding if and only if $V^1$ and $V^2$ are transitive tournaments with $n_1=\left\lceil\frac{n}{2}\right\rceil$ and $n_2=\left\lfloor\frac{n}{2}\right\rfloor$.
\end{proof}

\section{Concluding remarks}\label{sec4}
In Section ~\ref{sec2.1} of this paper, we fully characterized the extremal digraphs with a fixed dichromatic number that attain the minimal and maximal second Laplacian spectral moment among all join digraphs in $\mathcal{G}_{n,r}$.
From Theorem~\ref{th:ch-2.5} and Theorem~\ref{th:ch-2.6}, we know that the extremal digraphs are isomorphic to $\bigvee_{i=1}^{r} V^i$, where each $V^i$ is an in-tree with $n_1=n-r+1$ and $n_2=\cdots=n_r=1$, and each $V^i$ is a transitive tournament with $n_i=\left\lceil\frac{n}{r}\right\rceil$ or $n_i=\left\lfloor\frac{n}{r}\right\rfloor$, respectively.

In addition, in Section~\ref{sec2.2} of this paper, we characterized the extremal digraphs  that attain the minimal and maximal second Laplacian spectral moment among all digraphs in $\mathcal{G}_{n,r}$. In particular, the extremal digraphs in Theorem~\ref{th:ch-2.10} for the minimal second Laplacian spectral moment differ considerably from those in Theorem~\ref{th:ch-2.5}, and required a different proof approach.

We were unable to provide such a full characterization of the extremal digraphs for the third Laplacian spectral moment. However, restricting ourselves to join digraphs we demonstrated that in-trees and transitive tournaments play a key role there, too. From Theorem~\ref{th:ch-3.3}, we know that also for the third Laplacian spectral moment the extremal join digraphs are isomorphic to $\bigvee_{i=1}^{r} V^i$, where each $V^i$ is either an in-tree (for attaining the minimum), or a transitive tournament (for the maximum), but we could not determine the optimum values of $n_i$, except for the case when $r=2$. For $r=2$, we obtained a full characterization in Corollary~\ref{co:ch-3.4}, showing exactly the same extremal join digraphs as for the second Laplacian spectral moment. We complete this section with some examples of the extremal join digraphs in $\mathcal{G}_{n,3}$ for the third Laplacian spectral moment, as shown in Table~\ref{ta:ch-2}. Here we used $G[n_1,n_2,n_3]$ to denote the digraph $V^1\vee V^2\vee V^3$, and we used Lemma~\ref{le:ch-3.2} and Theorem~\ref{th:ch-3.3} to determine the extremal join digraphs.

\begin{table}[ht]
\renewcommand\arraystretch{1.5}\centering
\caption{The extremal join digraphs in $\mathcal{G}_{n,3}$ for the third Laplacian spectral moment}\label{ta:ch-2}
\begin{tabular}{|c|c|c|}
\hline
\tabincell{c}{$\mathcal{G}_{n,3}$} & \tabincell{c}{The minimal join digraphs for the third \\ Laplacian spectral moment \\
when each $V^i$ is an in-tree}
& \tabincell{c}{The maximal join digraphs for the third \\ Laplacian spectral moment when \\ each $V^i$ is a transitive tournament}
\\ \hline
$\mathcal{G}_{5,3}$ & $G[3,1,1]$ & $G[2,2,1]$ \\ \hline
$\mathcal{G}_{6,3}$ & $G[4,1,1]$ & $G[2,2,2]$ \\ \hline
$\mathcal{G}_{7,3}$ & $G[5,1,1]$ & $G[3,2,2]$ \\ \hline
$\mathcal{G}_{8,3}$ & $G[6,1,1]$ & $G[3,3,2]$ \\ \hline
$\mathcal{G}_{9,3}$ & $G[7,1,1]$ & $G[3,3,3]$ \\ \hline
$\mathcal{G}_{10,3}$ & $G[8,1,1]$ & $G[4,3,3]$ \\ \hline
$\mathcal{G}_{15,3}$ & $G[13,1,1]$ & $G[5,5,5]$ \\ \hline
$\mathcal{G}_{20,3}$ & $G[18,1,1]$ & $G[7,7,6]$ \\
\hline
\end{tabular}
\end{table}

From Table~\ref{ta:ch-2}, we conclude that the extremal join digraphs in $\mathcal{G}_{n,3}$ for the third Laplacian spectral moment are the same as for the second Laplacian spectral moment for $n=5,6,7,8,9,10,15,20$. This might suggest that the extremal join digraphs in $\mathcal{G}_{n,3}$ for the second and third Laplacian spectral moment are the same for all values of $n$, but we were unable to confirm this. We leave the full characterization of the extremal join digraphs in $\mathcal{G}_{n,r}$ for the third Laplacian spectral moment as an open problem.

\noindent\begin{problem}\label{pr:ch-4.1} Characterize the extremal digraphs for the third Laplacian spectral moment among all join digraphs with a fixed dichromatic number.
\end{problem}

We have partially solved the above problem by narrowing the extremal digraphs down to join digraphs $\bigvee_{i=1}^r V^i$, where either each $V^i$ is an in-tree or each $V^i$ is a transitive tournament. In this paper, we did not consider the $k$-th Laplacian spectral moment for higher values of $k\geq4$. We leave this as another challenging open problem.

\noindent\begin{problem}\label{pr:ch-4.2} Characterize the extremal digraphs for the $k$-th Laplacian spectral moment among all join digraphs with a fixed dichromatic number, for a fixed integer $k\ge 4$.
\end{problem}

We also leave the characterization for general digraphs with a fixed dichromatic number as an open problem. In the light of Theorem~\ref{th:ch-2.10}, we think this could be particularly challenging for the minimal spectral moments.

\noindent\begin{problem}\label{pr:ch-4.3} Characterize the extremal digraphs for the $k$-th Laplacian spectral moment among all digraphs with a fixed dichromatic number, for a fixed integer $k\ge 3$.
\end{problem}

\end{document}